\documentclass[10p]{llncs}
\usepackage[numbers]{natbib}
\usepackage{footmisc}
\usepackage{times}
\usepackage{hyperref}

\usepackage{amssymb}
\usepackage{amsmath}
\usepackage{algorithm}
\usepackage{algorithmic}
\usepackage{graphicx}
\usepackage{caption}
\usepackage{wrapfig}
\newtheorem{assumption}{Assumption}
\begin{document}
\title{ Convergence and Applications of ADMM on the Multi-convex Problems}
\author{Junxiang Wang   \and Liang Zhao\\\email{jwan936@emory.edu} \quad\quad \email{lzhao41@emory.edu}}
\institute{
Emory University, 201 Dowman Drive, Atlanta, Georgia, USA
}
\maketitle
\begin{abstract}
  In recent years, although the Alternating Direction Method of Multipliers (ADMM) has been empirically applied widely to many multi-convex applications, delivering an impressive performance in areas such as nonnegative matrix factorization and sparse dictionary learning, there remains a dearth of generic work on proposed ADMM with a convergence guarantee under mild conditions. In this paper, we propose a generic ADMM framework with multiple coupled variables in both objective and constraints. Convergence to a Nash point is proven with a sublinear convergence rate $o(1/k)$. Two important applications are discussed as special cases under our proposed ADMM framework. Extensive experiments on ten real-world datasets demonstrate the proposed framework's effectiveness, scalability, and convergence properties. We have released our code at \url{https://github.com/xianggebenben/miADMM}.
\end{abstract}
\section{Introduction}
\indent Due to the advantages and popularity of non-differentiable regularized and distributive computing for complex optimization problems, the Alternating Direction Method of Multipliers (ADMM) has received a great deal of attention in recent years \cite{boyd2011distributed}. The standard ADMM was originally proposed to solve the following separable convex optimization problem:
\begin{align*}
    \min\nolimits_{x,z} f(x)+g(z)\ \ \ s.t. \ Ax+Bz=c.
\end{align*}
where $f(x)$ and $g(z)$ are closed convex functions, $A$ and $B$ are matrices and $c$ is a vector. There are extensive reports in the literature exploring the theoretical properties for convex optimization problems related to ADMM and its variants, including multi-block ADMM \cite{deng2017parallel}, Bregman ADMM \cite{wang2014bregman}, fast ADMM \cite{goldstein2014fast,kadkhodaie2015accelerated}, and stochastic ADMM \cite{ouyang2013stochastic}. ADMM has now been extended to cover a wide range of nonconvex problems, and it has achieved outstanding performance in many practical applications \cite{xu2016empirical}. 

Unlike convex problems, the convergence theory on the nonconvex  ADMM is much more difficult, and considerable progress has been made on this problem, please refer to Section \ref{sec:related_work} for a detailed summary.
Recently, however, there has been an increasing number of real-world applications where the objective functions are multi-convex (i.e. nonconvex for all the variables but convex for each when all the others are fixed).  For example, nonnegative matrix factorization, which aims to decompose a matrix into a product of two matrices, has been applied widely in computer vision, machine learning, and various other fields \cite{lee2001algorithms}; A bilinear matrix inequality problem has been designed for the analysis of linear and nonlinear uncertain systems \cite{hassibi1999path}. \\
\indent All of the above applications can be considered as special cases of multi-convex optimization problems. However,  such problems have not yet been rigorously and systematically investigated by ADMM. Moreover, the convergence properties of the ADMM required to solve such problems remain unknown. In this work, we propose mild conditions to ensure the convergence of ADMM to a Nash point on the multi-convex problems with a sublinear convergence rate $o(1/k)$. We also discuss how our ADMM is applied to two important applications. Extensive experiments show the effectiveness of our proposed ADMM.  Our contributions in this paper include:
\begin{itemize}
\item  We propose an ADMM framework to solve the multi-convex problem, and we investigate the convergence properties of the proposed ADMM. Specifically, we prove that the objective value and the residual are convergent. Moreover, any limit point generated by the proposed ADMM is a Nash point of the original problem. The convergence rate of the proposed ADMM is $o(1/k)$.
\item  We demonstrate two important and promising applications that are special cases of our proposed ADMM framework and benefit from its theoretical properties. Specifically, we show how these applications can be transformed equivalently to fit into the proposed ADMM framework.
\item We conduct extensive experiments to validate our proposed ADMM. Experiments on ten real-world datasets demonstrate its effectiveness, scalability, and convergence properties.
\end{itemize}
\indent The rest of this paper is summarized as follows: Section \ref{sec:related_work} summarizes previous work related to this paper. Section \ref{sec:proposed ADMM} introduces the ADMM algorithm and its convergence properties. In Section \ref{sec:application}, the proposed ADMM algorithm is applied to several important applications. Extensive experiments are described in Section \ref{sec:experiment}. The paper concludes with a summary of the work in Section \ref{sec:conclusion}.
\section{Related Work}
\label{sec:related_work}
\textbf{Multi-convex optimization problems:} Some works studied multi-convex problems. The earliest work required that the objective function was differentiable continuous and strictly convex \cite{warga1963minimizing}. Various conditions on separability and regularity on the objective functions have been discussed in \cite{tseng1993dual,tseng2001convergence}. In the most recent work, Xu and Yin presented three types of multi-convex algorithms and analyzed convergence based on either Lipschitz differentiability or strong convexity assumption  \cite{xu2013block}. For a comprehensive survey, please refer to \cite{shen2017disciplined}.\\
\textbf{Convergence analysis of
 ADMM:} Existing literature on the convergence analysis of ADMM can be categorized into two classes: the convex ADMM and the nonconvex ADMM. The convex ADMM is investigated relatively well compared with the nonconvex ADMM. Existing works either study suitable stepsizes of the convex ADMM or extend ADMM to the stochastic version. For example, Bai et al. proposed a generalized symmetric ADMM to solve the multi-block separable objective by updating the Lagrange multiplier twice with suitable stepsizes \cite{bai2018generalized}; Gu et al. extended contractive Peaceman-Rachford splitting method to ADMM with larger stepsizes \cite{gu2015semi}; Ouyang et al. proposed a stochastic ADMM with a convergence rate $O(\frac{1}{\sqrt{t}})$.  Despite the outstanding performance of the nonconvex ADMM, its convergence theory is not well established due to the complexity of both coupled objectives and various (inequality and equality) constraints. Most existing works discussed the convergence of the nonconvex ADMM on separable objectives: they provided convergence guarantee to the stationary solutions with different assumptions \cite{boct2020proximal,chao2021inertial,chartrand2013nonconvex,li2015global}. Some works explored more difficult cases where the objectives are coupled: for example, Wang et al. presented mild convergence conditions of the nonconvex ADMM where the objective can be nonsmooth \cite{wang2019global}; Gao et al. explored the convergence condition of ADMM on multi-affine constraints \cite{gao2020admm};  Wang et al. gave the convergence proofs of ADMM in the nonconvex deep learning problems \cite{wang2020toward,wang2021towards,wang2019admm}; while experiments by Wang and Zhao showed that the ADMM was not necessarily convergent in the nonlinear-constrained problems \cite{wang2017nonconvex}.
\section{ADMM on the Multi-convex Problems}
\label{sec:proposed ADMM}
\indent In this section, we present an ADMM framework to solve Problem 1.
\subsection{Preliminaries}
\label{sec:preliminary}
\indent First, the definition of Lipschitz differentiability is shown as follows \cite{cavalletti2016tangent}:
\begin{definition}[Lipschitz Differentiability] Any arbitrary differentiable function
$G_1: \mathbb{R}^m\rightarrow \mathbb{R}$ is Lipschitz differentiable if for any $x^{'}, x^{''}\in \mathbb{R}^m$,
\begin{align*}
    \Vert \nabla G_1(x^{'})-\nabla G_1(x^{''})\Vert \leq  D \Vert x^{'}-x^{''}\Vert.
\end{align*}
where $D\geq 0$ is constant and $\nabla G_1(x)$ denotes the gradient of $G_1(x)$. 
\end{definition}
The following defines strong convexity, which is indispensable for the proof of convergence to a Nash point. 
\begin{definition}[Strong Convexity]
A convex function $G_4(x)$ is strongly convex if there exists $H>0$ such that for $\forall x^{'},x^{''}\in dom(G_4)$, the following holds
\begin{align*}
    G_4(x^{''})\geq G_4(x^{'})+(v^{'})^T(x^{''}-x^{'})+(H/2)\Vert x^{''}-x^{'}\Vert^2_2.
\end{align*}
where $\forall v^{'}\in\partial G_4(x^{'})$ is a subdifferential of $G_4$ at $x^{'}$.
\end{definition}
Finally, the Nash point is defined as follows \cite{xu2013block}:
\begin{definition}[Nash Point]
Given $G_5(a_1,a_2,\cdots,a_m)$, a Nash point $(a^*_1,a^*_2,\cdots,a^*_m)$ satisfies the following property:
\begin{align*}
    &G_5(a^*_1,\cdots, a^*_{i-1},a^*_i, a^*_{i+1},\cdots,a^*_m)\leq     G_5(a^*_1,\cdots, a^*_{i-1},a_i,a^*_{i+1},\cdots,a^*_m), \\&\forall (a^*_1,\cdots, a^*_{i-1},a_i,a^*_{i+1},\cdots,a^*_m)\in dom(G_5), \ (i=1,\cdots,m).
\end{align*}
\end{definition}
\indent Naturally, when we optimize one variable while fixing others, the Nash point ensures the optimality of this variable \cite{xu2013block}. Without loss of generality, we assume that Problem 1 has at least a Nash point, and in the next section, we will prove that any limit point generated by ADMM converges to a Nash point.
\subsection{The ADMM algorithm}
\label{sec:algorithm}
The following problem is our focus in this paper:
\begin{problem}
\label{prob:main problem}
  \begin{align*}\nonumber \min\nolimits_{x_1,\cdots,x_n,z} F(x_1,\cdots,x_n,z)=f(x_1,\cdots,x_n)\!+h(z)\ \ \  s.t. \sum\nolimits_{i=1}^n A_i x_i-z=0.
\end{align*}
\end{problem}
where $x_i\in \mathbb{R}^{p_i}(i=1,\cdots,n)$, $z\in \mathbb{R}^{q}$, $f(x_1,\cdots,x_n): \mathbb{R}^p\rightarrow \mathbb{R}\cup\{\infty\}(p=\sum\nolimits_{i=1}^n p_i)$ are proper, continuous, multi-convex and possibly nonsmooth functions, $h(z)$ is a proper, differentiable and convex function. $A_i\in \mathbb{R}^{q\times p_i}(i=1,\cdots,n)$ are matrices. Obviously, the domain of $F$ is $dom(F)=\{(x_1,\cdots, x_n,z)| \ \sum\nolimits_{i=1}^n A_i x_i-z=0\}$. Without the loss of generality, the objective of Problem 1 is assumed to be bounded from below. \\
\indent To ensure the convergence of the proposed ADMM, some mild assumptions are imposed, which are shown as follows:
\begin{assumption}[Lipschitz Differentiability]
$h(z)$ is Lipschitz differentiable with constant $H\geq 0$.
\label{ass:lipschitz differentiability and Coercivity}
\end{assumption}
\indent  Most loss functions such as the cross-entropy loss and the square loss are Lipschitz differentiable \cite{wang2019admm}. In order to propose the  ADMM algorithm, the augmented Lagrangian function can be formulated mathematically as follows:
\begin{align}
     L_\rho(x_1,\!\cdots,\!x_n,\!z,\!y)&=F(x_1,\!\cdots,\!x_n,z)\!+\!y^T(\sum\nolimits_{i\!=\!1}^n A_i x_i\!-\!z)\!+\!(\rho/2)\Vert \sum\nolimits_{i\!=\!1}^n A_i x_i\!-\!z \Vert^2_2. \label{eq:lagrangian function}
\end{align}
where $y$ is a dual variable and $\rho>0$ is a penalty parameter. The proposed ADMM aims to optimize the following $n+1$ subproblems alternately.
\begin{align}
    \nonumber  &x_i^{k+1} \leftarrow \arg\min\nolimits_{x_i} f(\cdots,x^{k+1}_{i-1},x_i,x^{k}_{i+1},\cdots)+(y^k)^TA_ix_i\\&+(\rho/2)\Vert\sum\nolimits_{j=1}^{i-1} A_j x^{k+1}_j+A_ix_i+\sum\nolimits_{j=i+1}^{n} A_jx^k_j-z^k\Vert^2_2.  \label{eq:update x}\\
    &z^{k+1} \leftarrow \arg\min\nolimits_{z} L_\rho(\cdots,x_n^{k+1},z,y^k)\label{eq:update z}\\\nonumber&\!=\!\arg\min\nolimits_z h(z)\!-\!(y^k)^Tz\!+\!(\rho/2)\Vert \sum\nolimits_{i\!=\!1}^n A_ix^{k+1}_i\!-\!z\Vert^2_2.
\end{align}
\begin{wrapfigure}{r}{0.5\linewidth}
\vspace{-1.5cm}
\begin{minipage}{\linewidth}
\begin{algorithm}[H]
\scriptsize
\caption{  The Proposed ADMM to Solve Problem 1} 
\begin{algorithmic}[1] 
\REQUIRE $A_i(i=1,\cdots,n),\delta>0$. 
\ENSURE $x_i(i=1,\cdots,n),z$. 
\STATE Initialize $\rho$, $k=0$.
\REPEAT
\FOR{i=1 to n}
\STATE Update $x^{k+1}_i$ in Equation \eqref{eq:update x}.
\ENDFOR
\STATE Update $z^{k+1}$ in Equation \eqref{eq:update z}.
\STATE $r^{k+1}\leftarrow \sum\nolimits_{i=1}^n A_ix_i^{k+1}-z^{k+1}$. \# update primal residual
\STATE $y^{k+1}\leftarrow y^k+\rho r^{k+1}$.\\
\STATE $k\leftarrow k+1$.
\UNTIL $\Vert r^{k+1}\Vert\leq \delta$.
\STATE Output $x_i(i=1,\cdots,n),z$.
\end{algorithmic}
\label{algo:proposed ADMM}
\end{algorithm}
\vspace{-1.5cm}
\end{minipage}
\end{wrapfigure}

Algorithm \ref{algo:proposed ADMM} is presented for Problem 1. Concretely, Lines 3-5 and 6 update $x^{k+1}_i(i=1,\cdots,n)$ and $z^{k+1}$, respectively. Line 7 updates the primal residual $r^{k+1}$, which is defined in accordance with the standard ADMM \cite{boyd2011distributed}: it measures how the linear constraint $\sum\nolimits_{i=1}^n A_ix_i-z=0$ is violated.  Line 8 updates the dual variable $y^{k+1}$, which follows the routine of the standard ADMM. Line 10 uses the norm of the primal residual $r$ as a condition to terminate the algorithm, where $\delta>0$ is a threshold. Each subproblem is convex and implicitly assumed to be solvable.
\subsection{Convergence Analysis}
\label{sec:convergence}




\indent This section focuses on the convergence of the proposed ADMM algorithm. Specifically, 
 the first lemma states that the augmented Lagrangian  $L_\rho$ keeps decreasing, which is stated as follows.
\begin{lemma}[Objective Descent] 
\label{lemma:objective descent}
If $\rho>2H$ so that $C_1=\rho/2-H/2-H^2/\rho>0$, then there exists $C_2=\min(\rho/2,C_1)$ such that
\begin{align}
\nonumber &L_\rho(x_1^k,\cdots,x_n^k,z^k,y^k)-L_\rho(x_1^{k+1},\cdots,x_n^{k+1},z^{k+1},y^{k+1})\\&\geq C_2(\Vert z^{k+1}-z^k\Vert^2_2+\sum\nolimits_{i=1}^n\Vert A_i(x_i^{k+1}-x_i^{k})\Vert^2_2).\label{ineq: objective descent} 
\end{align}
\end{lemma}
\indent   Lemma \ref{lemma:objective descent} holds under Assumption \ref{ass:lipschitz differentiability and Coercivity}, and its proof can be found in Section \ref{sec: proof of objective descent and objective bound} in the supplementary materials\footnote{The supplementary materials are available at \url{https://github.com/xianggebenben/miADMM/blob/main/multi_convex_ADMM-13-18.pdf}\label{fn:supplementary}}. The next lemma states that the augmented Lagrangian is bounded from below, as shown below:
\begin{lemma}[Objective Bound]
\label{lemma:objective bound} If $\rho> 2H$, then 
$L_\rho(x_1^k,\cdots,x^k_n,z^k,y^k)$ is lower bounded.
\end{lemma}
  The proof of Lemma \ref{lemma:objective bound} can be found in Section \ref{sec: proof of objective descent and objective bound} in the supplementary materials \footref{fn:supplementary}.
 Now we can prove that the proposed ADMM converges globally in the following theorem.
\begin{theorem}[Residual and Objective Convergence]
If $\rho>2H$, then for the bounded sequence $(x_1^k,\cdots, x_n^k,z^k,y^k)$, then it has the following properties:\\
a). Residual convergence. This means that as $k \rightarrow \infty$, $ r^{k}\rightarrow 0$, where $r^k$  is defined in Algorithm \ref{algo:proposed ADMM}.\\
b). Objective convergence. This means that as $k\rightarrow \infty$, $F(x^k_1,\cdots,x^k_n,z^k)$ converges.
\label{thero: theorem 2}
\end{theorem}

\indent Theorem \ref{thero: theorem 2} guarantees the convergence of the proposed ADMM, whose proof is in Section \ref{sec:convergence proof} in the supplementary materials \footref{fn:supplementary}. However, $x^k_i(i=1,\cdots,n)$ and $z^k$ are not necessarily shown to be convergent. The next theorem states that any limit point is a feasible Nash Point of Problem 1.
\begin{theorem}[Convergence to a Nash Point]
Let $\rho> 2H$, if \textbf{either} of two assumptions hold: (a). $A_i(i=1,\cdots,n)$ have full rank. (b). $F$ is strongly convex with regard to $x_i$.
Then for bounded variables $(x^k_1,\cdots,x^k_n,z^k)$, it has at least a limit point $(x^*_1,\cdots,x^*_n,z^*)$, and any limit point $(x^*_1,\cdots,x^*_n,z^*)$ is a  feasible Nash point of $F$ defined in Problem 1. That is
\begin{align*}
    &\sum\nolimits A_ix^*_i-z^*=0. \ \text{(feasibility)}\\
    &F(x^*_1,\cdots,x^*_n,z^*)\leq    F(x^*_1,\cdots, x^*_{i-1},x_i,x^*_{i+1},\cdots,x^*_n,z^*),\\& \forall (x^*_1,\cdots, x^*_{i-1},x_i,x^*_{i+1},\cdots,x^*_n,z^*)\in dom(F),(i=1,\cdots,n).\\&
    F(x^*_1,\cdots,x^*_n,z^*)\leq     F(x^*_1,\cdots, x^*_n,z) \forall (x^*_1,\cdots,x^*_n,z)\in dom(F) \ \text{(Nash point)}.
\end{align*}
\label{thero: Nash point theorem}
\end{theorem}
\vspace{-0.5cm}
\indent The proof of Theorem \ref{thero: Nash point theorem} is detailed in Section \ref{sec:convergence proof} in the supplementary materials \footref{fn:supplementary}.
The third theorem states that our proposed ADMM can achieve a sublinear convergence rate of $o(1/k)$ under Assumption \ref{ass:lipschitz differentiability and Coercivity}, despite the nonconvex and complex nature of Problem 1. Such a rate is state-of-the-art even compared to those methods for simpler convex problems. The theorem is shown as follows:
\begin{theorem}[Convergence Rate]
If $\rho>2H$, for a bounded sequence $(x^k_1,\cdots,x^k_n,z^k,y^k)$,  define $u_k=\min\nolimits_{0\leq l\leq k}(\Vert z^{l+1}-z^l\Vert^2_2+\sum\nolimits_{i=1}^n\Vert A_i(x_i^{l+1}-x_i^{l})\Vert^2_2)$, then the convergence rate of $u_k$ is $o(1/k)$.
\label{thero: theorem 3}
\end{theorem}
\indent The proof of this theorem is in Section \ref{sec:convergence proof} in the supplementary materials \footref{fn:supplementary}. The $o(1/k)$ convergence rate of the proposed ADMM is consistent with much existing work analyzing the convex ADMM, including \cite{deng2017parallel,he20121,lin2015sublinear}. Our contribution in term of convergence rate is that we extend the guarantee of $o(1/k)$ into the multi-convex Problem 1.\\
\indent Our proposed ADMM is more general than some influential works in terms of formulation. The relations between our proposed ADMM and previous works are summarized as follows:\\
\textbf{1. Generalization of Block Coordinate Descent (BCD) for multi-convex problems.} When the linear constraint $\sum\nolimits_{i=1}^n A_ix_i=z$ is removed in Problem 1, then the proposed ADMM is reduced to the Block Coordinate Descent  \cite{xu2013block}.\\
\textbf{2. Generalization of multi-block ADMM.}
When $f(x_1,\cdots,x_n)=0$, the proposed ADMM is reduced to the convex multi-block ADMM \cite{tao2016convergence}, i.e. the ADMM with no less than three variables.\\
\indent Apart from general formulations, the convergence guarantees of our proposed ADMM cover more applications than previous literature. For example,  \cite{wang2019global} requires the coupled objective $f(x_1,\cdots,x_n)$ to be Lipschitz differentiable. However, some important applications such as weakly-constrained multi-task learning (Section \ref{sec:weak_multitask}) and learning with signed-network constraints (Section \ref{sec:network_constraints}) do not satisfy this condition. But they are covered by our convergence guarantees of the multi-convex ADMM to a Nash point.
\section{Applications}
\label{sec:application}
\indent In this section, we apply our proposed ADMM to two real-world applications, both of which conform to Problem 1 and benefit from the convergence properties of the proposed ADMM.


\subsection{Weakly-constrained Multi-task Learning}
\label{sec:weak_multitask}
{ In multi-task learning problems, multiple tasks are learned jointly to achieve a better performance compared with learning tasks independently \cite{wang2018incomplete}. Most work on multi-task learning has tended to enforce the assumption of similarity among the feature weight values across tasks \cite{argyriou2007multi,chen2011integrating,wang2017multi,wang2018incomplete,zhou2011malsar} because this makes it possible to use convex regularization terms like $\ell_{2,1}$ norms \cite{wang2017multi} and Graph Laplacians \cite{zhou2011malsar}. However, this assumption is usually too strong and is seldom satisfied by the real-world data. Instead of requiring feature weights to be similar in magnitude, a more conservative but probably more reasonable assumption is that multiple tasks share similar polarities for the same feature, which means that if a feature is positively relevant to the output of a task, then its weight will also be positive for other related tasks. This assumption is appropriate for many applications. For example, the feature `number of clinic visits' will be positively related to flu outbreaks, while the feature `popularity of vaccination' will be negatively related to them, even though their feature weights can vary dramatically for different countries (namely tasks here). This is achieved by enforcing the requirement for every pair of tasks with neighboring indices to have the same weight sign. This optimization objective is shown as follows:}
\begin{align}
    &\min\nolimits_{w_1,\cdots,w_n}\sum\nolimits_{i=1}^n (Loss_i(w_i)+\Omega_i(w_i)) \label{eq: weakly-constrained multi-task learning}\\ \nonumber &s.t. \ w_{i,j}w_{i+1,j}\geq 0 \ (i=1,2,\cdots,n-1, j=1,2,\cdots,m).
\end{align}
where $n$ and $m$ denote the number of tasks and features, respectively, $w_{i,j}$ is the weight of the $j$-th feature in the $i$-th task, $w_i$ is the weight of the $i$-th task, and $Loss_i(w_i)$ and $\Omega_i(w_i)$ are the loss function and the regularization term of the $i$-th task, respectively.  The inequality constraint implies that the $i$-th  and the $i+1$-th tasks share the same sign for their weights. Equation \eqref{eq: weakly-constrained multi-task learning} is rewritten in the following form to fit in our proposed ADMM framework:
\begin{align}
 &\min\nolimits_{w_1,\cdots,w_n,z} \sum\nolimits_{i=1}^n (Loss_i(w_i)+\Omega_i(z_i))+\lambda_1\sum\nolimits_{i=1}^{n-1}\sum\nolimits_{j=1}^{m}c_1(w_{i,j}w_{i+1,j}) \label{prob: multi-task learning}\\
    &\nonumber s.t. \ z_i=w_i \ (i=1,2,\cdots,n).
\end{align}
where $z=[z_1;\cdots;z_n]$ is an auxiliary variable, and $\lambda_1>0$ is a tuning parameter. Notice that the inequality constraint $w_{i,j}w_{i+1,y}\geq 0$ is transformed  to a quadratic penalty $c_1(x)$ such that
$
    c_1(x)=
    \begin{cases}
    x^2 & x<0\\
    0 & x\geq 0
    \end{cases}
$
 which makes the formulation consistent with Problem \ref{prob:main problem}. The proposed ADMM algorithm for this case is shown in Appendix \ref{sec:multi-task learning} in the supplementary materials \footref{fn:supplementary}.
\subsection{Learning with Signed-Network Constraints}
\label{sec:network_constraints}
\indent The application of network models for social network analysis has attracted the attention of a large number of researchers \cite{carrington2005models}.  For example, influential societal events often spread across many social networking sites and are expressed in different languages. Such multi-lingual indicators usually transmit similar semantic information through networks and have thus been utilized to facilitate social event forecasting \cite{zhao2018distant}. The problem with network constraints is formulated as follows:
\begin{gather*}
    \min\nolimits_{\beta_1,\cdots,\beta_n} Loss(\beta_1,\cdots,\beta_n)+\sum\nolimits_{i=1}^n \omega_i(\beta_i)\\
     s.t. \  \exists (\beta_i,\beta_j)\in E_s, \exists (\beta_p,\beta_q)\in  E_d \ (1\leq i,j,p,q\leq n).
\end{gather*}
where $\beta_i$ is the weight of the $i$-th node. $Loss(\beta_1,\cdots,\beta_n)$ is a loss function and $\omega_i(\beta_i)$ is a regularization term for the $i$-th node. $E_s=\{(\beta_i,\beta_j)|\beta_i \beta_j\geq 0\}$ and $E_d=\{(\beta_p,\beta_q)|\beta_p \beta_q\leq 0\}$ are two edge sets to represent two opposite relationships: $(\beta_i,\beta_j) \in E_s$ means that $\beta_i\beta_j\geq 0$, while $(\beta_p,\beta_q) \in E_d$ means that  $\beta_p\beta_q\leq 0$. The constraint means that some pair  $(\beta_i,\beta_j)$ satisfies the edge set $E_s$, and that some pair  $(\beta_p,\beta_q)$ satisfies the edge set $E_d$. For example, in the problem of  social event forecasting with French and English, $E_s$ and $E_d$ are edge sets of synonyms and antonyms between French and English, and the weight pair of the French word "bien" and the English word "good" belongs to $E_s$. The problem with network constraints can be reformulated approximately
to the following:
\begin{align}
 \nonumber&\min\nolimits_{\beta_1,\cdots,\beta_n,z} Loss(\beta_1,\cdots,\beta_n)+\sum\nolimits_{i=1}^n \omega_i(z_i)+\lambda_2(\sum\nolimits_{(\beta_i,\beta_j)\in E_s} c_2(\beta_i,\beta_j)\\&+\sum\nolimits_{(\beta_p,\beta_q)\in E_d} c_3(\beta_p,\beta_q)) s.t. \ z_i=\beta_i \ (i=1,2,\cdots,n)\label{prob: muli-lingual}
\end{align}
where $z=[z_1;\cdots;z_n]$ is an auxiliary variable, and $\lambda_2>0$ is a tuning parameter. The constraint $(\beta_i,\beta_j)\in E_s$ and $(\beta_p,\beta_q)\in  E_d (1\leq i,j,p,q\leq n)$ are transformed to two quadratic penalties $c_2(\beta_i,\beta_j)$ and $c_3(\beta_p,\beta_q)$ as follows:\\
$c_2(\beta_i,\beta_j)=
    \begin{cases}
    (\beta_i\beta_j)^2 & (\beta_i,\beta_j)\not\in E_s\\
    0 & (\beta_i,\beta_j)\in E_s
    \end{cases}, c_3(\beta_p,\beta_q)=
    \begin{cases}
    (\beta_p\beta_q)^2 & (\beta_p,\beta_q)\not\in E_d\\
    0 & (\beta_p,\beta_q)\in E_d\end{cases}
$.\\
The proposed ADMM for this case is also shown in Appendix \ref{sec:muli-lingual} in the supplementary materials \footref{fn:supplementary}. 
\section{Experiments}
\label{sec:experiment}
In this section, we test our proposed ADMM using ten real-world datasets on two applications detailed in Section \ref{sec:application}.  Scalability, effectiveness, and convergence properties are compared with several existing state-of-the-art methods on ten real datasets. All experiments were conducted on a 64-bit machine with Intel(R) Core(TM)  processor (i7-6820HQ CPU@ 2.70GHZ) and 16.0GB memory.
 \begin{figure*}[h]
 \vspace{-0.5cm}
 \centering
 \begin{minipage}{0.24\linewidth}
   \includegraphics[width=\textwidth]{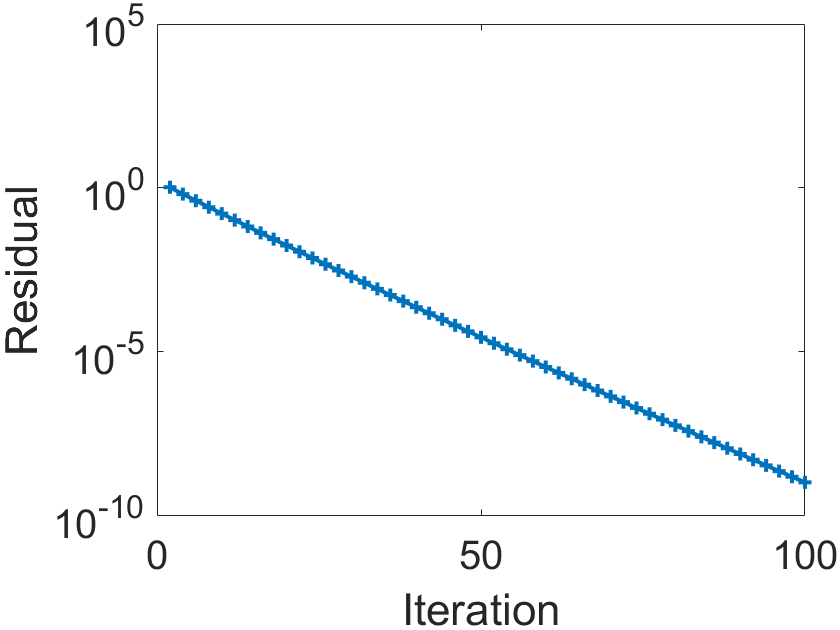}
   \centerline{(a). Residual on}
    \centerline{Experiment I.}

 \end{minipage}
 \hfill
 \begin{minipage}{0.24\linewidth}
   \includegraphics[width=\textwidth]{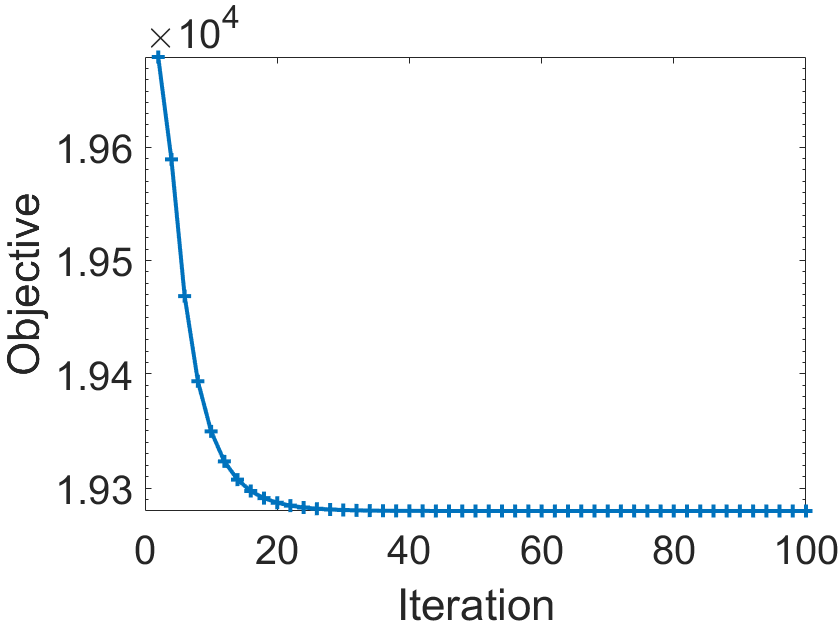}
    \centerline{(b). Objective on}
    \centerline{Experiment I.}
 \end{minipage}
 \hfill
 \begin{minipage}{0.24\linewidth}
   \includegraphics[width=\textwidth]{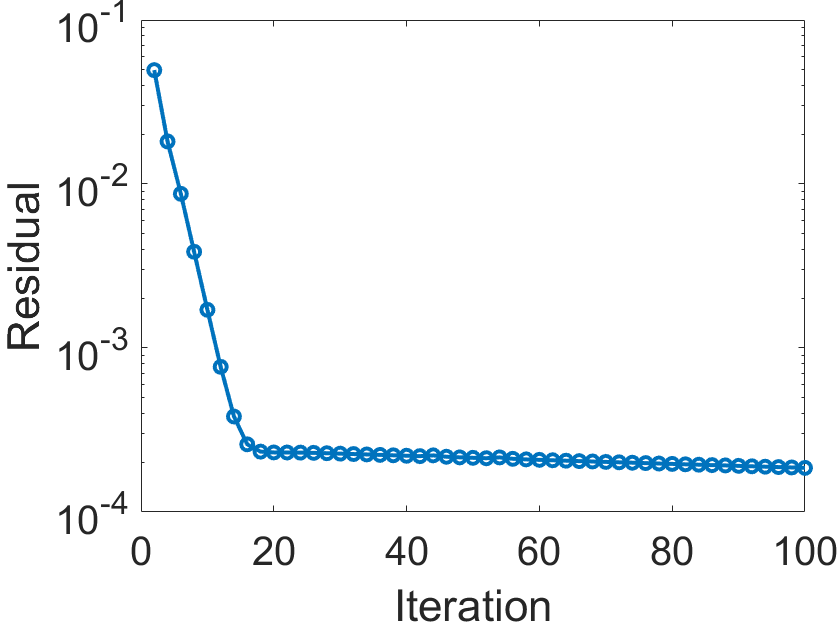}
   \centerline{(c). Residual on the VE} 
   \centerline{dataset of}
   \centerline{Experiment II.}
 \end{minipage}
 \hfill
 \begin{minipage}{0.24\linewidth}
   \includegraphics[width=\textwidth]{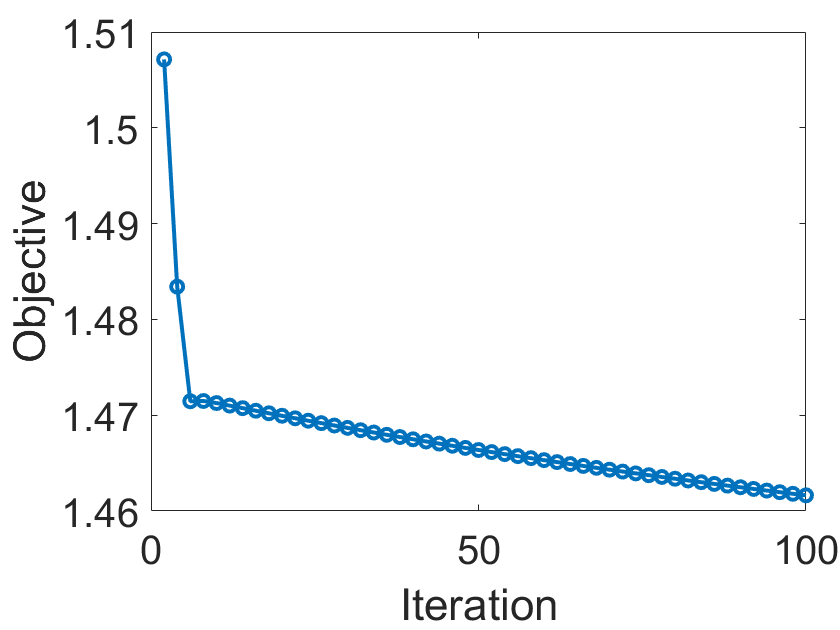}
   \centerline{(d). Objective on the VE}
   \centerline{dataset of}
   \centerline{Experiment II.}
 \end{minipage}
  \vspace{-0.3cm}
 \caption{Convergence curves on Experiments I and II.}
  \vspace{-1cm}
 \label{fig:convergences}
\end{figure*}

\subsection{Experiment I: Weak-constrained Multi-task Learning}
To evaluate the effectiveness of our method on the application of weak-constrained multi-task learning described in Equation \eqref{prob: multi-task learning}, a real-world school dataset is used to evaluate the effectiveness of our proposed ADMM. It consists of the examination scores in three years of 15,362 students
from 139 secondary schools, which are treated as tasks for examination
scores prediction based on 27 input features such as year of the examination, school-specific features, and student-specific features. The dataset is publicly available and the detailed description can be found in the original paper \cite{li2015multi}. $\rho$ was set to $1000$. Here we chose two kinds of $\lambda_1$: (1) $\lambda_1^k=10^5$; (2) $\lambda^{k+1}_1=\lambda^k_1+10$ with $\lambda_1^k=1$. $\lambda_1(1)$ and $\lambda_1(2)$ are the first and the second choice of $\lambda_1$, respectively.\\
\textbf{Metrics.} In this experiment, five metrics were utilized to evaluate model performance. Mean Squared Error (MSE) measures the average of the squares of the difference between observation and estimation. Different from MSE, Mean Squared Logarithmic Error (MSLE) measures the ratio of observation to estimation. Mean Absolute Error (MAE) is also an error measurement but computed in the absolute value. The less the above three metrics are, the better a regression model is. Explained Variance (EV) computes the ratio of the variance of the error to that of observation. The coefficient of determination or R2 score is the proportion of the variance in the dependent variable that is predictable from the independent variable. The higher score of EV and R2 are, the better a regression model is.\\
\textbf{Baselines.} To validate the effectiveness of the proposed ADMM, five benchmark multi-task learning models served as comparison methods. Loss functions were set to least square errors. The number of iterations was set to $5,000$. The regularization parameter $\alpha$ was set based on 5-fold cross-validation on the training set.\\
\indent 1. multi-task learning with Joint Feature Selection (JFS) \cite{argyriou2007multi,zhou2011malsar} . JFS is one of the most commonly used strategies in multi-task learning. It captures the relatedness of multiple tasks by a constraint of a weight matrix to share a common set of features. $\alpha$ was set to 100.  \\
\indent 2. Clustered Multi-Task Learning (CMTL) \cite{zhou2011clustered,zhou2011malsar}. CMTL assumes that multiple tasks are clustered into several groups. Tasks in the same group are similar to each other. $\alpha$ was set to 1.\\
\indent 3. multi-task Lasso (mtLasso) \cite{zhou2011malsar}. mtLasso extends the classic Lasso model to the multi-task learning setting. $\alpha$ was set to 10.\\
\indent 4. a convex relaxation of Alternating Structure Optimization (cASO) \cite{zhou2011malsar,ando2005framework}. cASO decomposes each task into two components: task-specific feature mapping and task-shared feature mapping. $\alpha$ was set to 0.01.\\
\indent 5. Block Coordinate Descent (BCD) \cite{xu2013block}. BCD is an intuitive method to solve multi-convex problems, which optimizes each variable alternately. $\alpha$ was set to 10.
\vspace{-0.5cm}
\begin{wraptable}{r}{0.5\linewidth}
\scriptsize
\centering
\vspace{-0.3cm}
\caption{Performance in Experiment I.}
\vspace{-0.3cm}
\begin{tabular}{c|c|c|c|c|c}
\hline\hline
\multicolumn{6}{c}{Mean}\\\hline
Method & MSE & MSLE&MAE &EV & R2 \\
\hline\hline
JFS&114.1052&    0.4531&    8.4349&    0.2948&    0.2948\\ \hline
CMTL&114.9892&     0.4647&     8.4756&     0.2876 &    0.2875
\\ \hline
mtLasso&115.3143&     0.4625&     8.4725&     0.2873&     0.2873

\\\hline
cASO&137.8336&     0.5204&     9.3450&     0.1606&     0.1605
\\\hline
BCD&149.2313&	0.5577&	9.8057&	0.1299&	0.0777\\\hline

 ADMM($\lambda_1(1)$)&113.6975&     \textbf{0.4423}&     8.4024&     0.2950&     0.2960\\\hline
 ADMM($\lambda_1(2)$)&\textbf{113.2400}&0.4428&\textbf{8.3943}&\textbf{0.3002}&\textbf{0.3002}
\\
    \hline\hline
    \multicolumn{6}{c}{Standard Deviation}\\\hline
Method & MSE & MSLE&MAE &EV & R2 \\
\hline\hline
JFS&2.02& 0.02&    0.06&    0.02&    0.02\\ \hline
CMTL&1.85&     0.02& 0.05&     \textbf{0.01} &    \textbf{0.01}\\ \hline
mtLasso&1.77&0.02&0.05&     \textbf{0.01}&\textbf{0.01}
\\\hline
cASO&7.26& 0.01& 0.06& \textbf{0.01}&     \textbf{0.01}\\\hline
BCD&1.41&	0.01&	0.06&	0.15&	\textbf{0.01}\\\hline

 ADMM($\lambda_1(1)$)&\textbf{0.83}&     \textbf{0.005}&     \textbf{0.03}&     \textbf{0.01}&     \textbf{0.01}\\\hline
 ADMM($\lambda_1(2)$)&0.95&     0.01&     0.04&     0.02&     0.02\\
    \hline\hline
\end{tabular}
\label{tab:multi-task performance}
\vspace{-1cm}
\end{wraptable}

\textbf{Performance.}
As discussed in Section \ref{sec:weak_multitask}, the convergence of our proposed ADMM is guaranteed based on our theoretical framework. To verify this, Figures \ref{fig:convergences}(a) and \ref{fig:convergences}(b) illustrate the residual and objective values in different iterations, which demonstrates the convergence of the proposed ADMM on this nonconvex problem. Then the performance of examination score prediction on this dataset is illustrated in Table \ref{tab:multi-task performance}. Table \ref{tab:multi-task performance} shows the mean and the standard deviation of all methods, which were repeated 10 times by initializing parameters randomly, to make experimental evaluation robust. It shows that $\lambda_1(2)$ outperforms $\lambda_1(1)$ in four out of five metrics for the proposed ADMM. In addition, the proposed ADMM achieves the best performance in all the metrics, compared to all comparison methods. Moreover, the standard deviation of the proposed ADMM is about $30\%$ smaller than any other comparison method. This is because our method only enforces that the sign of the feature weight across different tasks is the same, while comparison methods typically perform too aggressive assumptions on the similarity among tasks. For example, CMTL enforces that the correlated tasks need to have similar feature weights using squared regularization on the difference between feature weights. JFS and mtLasso still tend to enforce similar weights on features in different tasks by $\ell_{2,1}$ norm. Because their enforcement is weaker than CMTL, their performance is better. cASO gets relatively weak performance because it optimizes an approximation of a nonconvex problem, and thus the approximate solution may be distant from the true solution to the original problem. Finally, the BCD performs the worst among all methods, even though it shares the same formulation with our proposed ADMM. This reflects the advantage of our proposed ADMM algorithm: dual information in one iteration can be passed to its following iteration by dual variables, which yields better performance. \\

\begin{wrapfigure}{r}{0.4\linewidth}
                \vspace{-1.7cm}
\begin{center}
    \includegraphics[width=\linewidth]{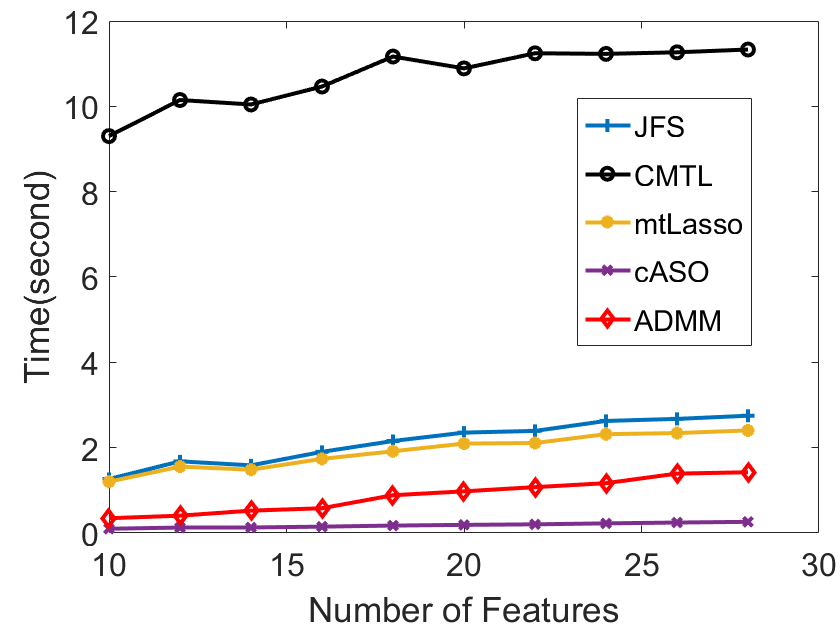}
 \caption{The training time of all methods in Experiment I.}
      \label{fig:scalability}
                \vspace{-1.cm}
          \end{center}
 \end{wrapfigure}
\textbf{Scalability.} To investigate the scalability of the proposed ADMM compared with all baselines in Experiment I, we measured the training time of them in the school dataset when the number of features varies. The training time was averaged by running 20 times. Figure \ref{fig:scalability} shows the training time of all methods when the number of features ranges from 10 to 28. The training time of all methods increased linearly concerning the number of features.  cASO was the most efficient of all methods, while the proposed ADMM was ranked second. mtLasso and JFS also trained a model within 5 seconds on average. CMTL was time-consuming for training, which spent more than 10 seconds. 
\vspace{-0.5cm}
\subsection{Experiment II: Event Forecasting with Multi-lingual Indicators}

\textbf{Datasets.}
To evaluate the performance of our proposed ADMM on the application in Section \ref{sec:network_constraints}, extensive experiments on nine real-world datasets have been performed. The dataset is obtained by randomly
sampling 10\% (by volume) of the Twitter data from Jan 2013 to Dec 2014. The data in the first and second years are used and training and test set, respectively. For the topic (i.e., social unrest) of interest, 1,806 keywords in the three major languages in Latin America, namely English, Spanish, and Portuguese, were provided by the paper \cite{zhao2018distant}. Their translation relationships have also been labeled as semantic links among them, such as ``protest'' in English, ``protesta'' in Spanish, and ``protesto'' in Portuguese. The event forecasting results were validated against a labeled event set, known as the Gold Standard Report (GSR), which is publicly available \cite{EN8FUW_2017}.\\
\textbf{Metric and Baselines.}
The metric used to evaluate the performance is Area Under the receiver operating characteristic Curve (AUC). Five comparison methods including the state-of-the-art Multi-Task learning (MTL), Multi-Resolution Event Forecasting (MREF), and Distant-supervision of Heterogeneous Multitask Learning (DHML) as well as classic methods Logistic Regression (LogReg) and Lasso. $\rho$ was set to $10$. Here we chose two kinds of $\lambda_2$: (1) $\lambda_2^k=10^5$; (2) $\lambda^{k+1}_2=\lambda^k_2+10$ with $\lambda_2^k=1$. $\lambda_2(1)$ and $\lambda_2(2)$ are the first and the second choice of $\lambda_2$, respectively. All the hyper-parameters were tuned by 5-fold cross-validation.\\ 
 \begin{wraptable}{r}{0.65\linewidth}
 \vspace{-1.3cm}
 \scriptsize
\centering
\caption{Event forecasting performance in AUC in each of the 9 datasets.}\label{tab:multi-lingual performance}
\vspace{-0.3cm}
\begin{tabular}{l|l|l|l|l|l|l|l|l|l}\hline\hline
&        BR & CL  & CO & EC & EL & MX & PY & UY & VE  \\\hline\hline
LogReg     & 0.686 & 0.677   & 0.644  & 0.599      & 0.618 & 0.661   & 0.616  & 0.628    & 0.667 \\
LASSO      & 0.685 & 0.677   & 0.648  & 0.603      & 0.636 & 0.665   & 0.615  & 0.666    & 0.669 \\
MTL        & 0.722 & 0.669   & 0.810  & 0.617      & 0.772 & 0.795   & 0.600  & 0.811    & 0.771 \\
MREF       & 0.714 & 0.563   & 0.515  & 0.784      & 0.612 & 0.693   & 0.658  & 0.681    & 0.588 \\
DHML      & 0.845       &  0.683             &  0.846          & 0.839      & 0.780 &  0.793 & 0.737  &  0.835  &  0.835\\
BCD &0.847&0.668&0.850&0.830&0.773&0.800&0.736&0.835&0.856\\
 ADMM $(\lambda_2(1))$&0.864&0.699&0.870&0.848&0.794&0.820&0.746&\textbf{0.850}&\textbf{0.867}\\ ADMM $(\lambda_2(2))$& 
 \textbf{0.867}&\textbf{0.701}&\textbf{0.872}&\textbf{0.851}&\textbf{0.798}&\textbf{0.823}&\textbf{0.747}&0.847&0.865\\
 \hline\hline
\end{tabular}
 \vspace{-0.7cm}
\end{wraptable}
\textbf{Performance.} As shown in Table \ref{tab:multi-lingual performance}, $\lambda_2(2)$ outperforms $\lambda_2(1)$ marginally in seven out of nine datasets for the proposed ADMM, and they generally perform the best among all the methods, with DHML and BCD the second-best performer. They all outperform the others typically by at least 5\%-10\%. This is because they leverage the multilingual correlation among the features to boost up the model's generalizability. Thanks to the framework of multi-task learning, MTL and MREF obtained a competitive performance with AUC typically over 0.7, which outperforms simple methods like LogReg and LASSO by 5\% on average.
\begin{wraptable}{r}{1.5\linewidth}
\vspace{-0.5cm}
\scriptsize
\centering
\begin{tabular}{l|l|l|l|l|l|l}\hline\hline
&LogReg&LASSO&MTL&MREF&DHML& ADMM\\\hline
BR&30193&1535&233&25889&332&\textbf{14}\\\hline CL&2981&242&35&6521&852&\textbf{11}\\\hline CO&8060&780&108&14714&87&\textbf{31}\\\hline EC&312&295&\textbf{17}&4332&46&25\\\hline EL&551&261&17&4669&33&\textbf{3}\\\hline MX&17712&2043&853&31349&175&\textbf{29}\\\hline PY&7297&527&40&9495&242&\textbf{5}\\\hline UY&748&336&20&5305&82&\textbf{3}\\\hline VE&5563&1008&49&5769&179&\textbf{28}\\\hline\hline
\end{tabular}
\centering
\vspace{-0.3cm}
\captionof{table}{Comparison of running time (in seconds) on 9 datasets in Experiment II.}
\label{tab:runtime}
\vspace{-0.9cm}
\end{wraptable}

\textbf{Efficiency.} In Experiment II, we also compared the training time of the proposed ADMM in comparison with all baselines on 9 datasets. The training time was averaged by running 5 times. The training time was shown in Table \ref{tab:runtime}. We do not show  BCD because its training time is similar to the proposed ADMM. Overall, the proposed ADMM was the most efficient of all methods for all datasets. It consumed no more than 30 seconds on all datasets. MTL  ranked second, but it spent hundreds of seconds on some datasets, like BR and MX. As the most time-consuming baselines, LogReg and MREF trained a model in thousands of seconds or more.
\vspace{-0.5cm}
\section{Conclusions}
\label{sec:conclusion}
\vspace{-0.3cm}
\indent We propose an ADMM framework for multi-convex problems with multiple coupled variables. It not only inherits the merits of general ADMMs but also provides advantageous theoretical properties on convergence conditions and properties under mild conditions. Besides, several machine learning applications of recent interest are discussed as special cases of our proposed ADMM. Extensive experiments have been conducted on ten real-world datasets, and demonstrate the effectiveness, scalability, and convergence properties of our proposed ADMM.
\vspace{-0.5cm}
\section*{Acknowledgement}
\vspace{-0.3cm}
This work was supported by the National Science Foundation (NSF) Grant No. 1755850, No. 1841520, No. 2007716, No. 2007976, No. 1942594, No. 1907805, a Jeffress Memorial Trust Award, Amazon Research Award, NVIDIA GPU Grant, and Design Knowledge Company (subcontract No: 10827.002.120.04).
\tiny
\bibliographystyle{plain}
\bibliography{reference}
\scriptsize
\onecolumn
\textbf{\large Appendix}
\begin{appendix}
\section{Preliminary Lemmas}
\label{sec:proofs}
\indent In this section, we give preliminary lemmas which are also used in the proof the proofs of Lemmas \ref{lemma:objective descent} and \ref{lemma:objective bound}. While Lemmas \ref{lemma: z optimal} and \ref{lemma: property 2 x} depend on the optimality conditions of subproblems, Lemmas \ref{lemma: lemma1}, \ref{lemma: property 2 y} and \ref{lemma:dual boundness} require Assumption \ref{ass:lipschitz differentiability and Coercivity}.
\begin{lemma}
\label{lemma: lemma1}
It holds that $\forall z_1,z_2\in R^q$,
\begin{align*}
    &h(z_1)\leq h(z_2)+\nabla h(z_2)^T(z_1-z_2)+(H/2)\Vert z_1-z_2\Vert^2, \ 
    -h(z_1)\leq -h(z_2)-\nabla h(z_2)^T(z_1-z_2)+(H/2)\Vert z_1-z_2\Vert^2.
\end{align*}
\end{lemma}
\begin{proof}
Because $h(z)$ is Lipschitz differentiable by Assumption \ref{ass:lipschitz differentiability and Coercivity}, so is $-h(z)$. Therefore, this lemma is proven exactly as same as Lemma 2.1 in \cite{beck2009fast}.
\end{proof}
\begin{lemma}
\label{lemma: z optimal}
It holds that
$y^k=\nabla h(z^k)$ for all $k\in \mathbb{N}$.
\end{lemma}
\begin{proof}
The optimality condition for the problem with regard to $z^{k}$ gives rise to
\begin{align*}
    \nabla h(z^k)-y^{k-1}-\rho (\sum\nolimits_{i=1}^n A_ix^k_i-z^k)=0.
\end{align*}
Because $y^k=y^{k-1}+\rho(\sum\nolimits A_ix_i^k-z^k)$, we have $y^k=\nabla h(z^k)$.
\end{proof}
\begin{lemma}
\label{lemma: property 2 x}
It holds that for $\forall k\in \mathbb{N}$,
\begin{align}
L_\rho(\cdots,x_{i-1}^{k+1},x_i^k,\cdots)-L_\rho(\cdots,x_i^{k+1},x_{i+1}^k,\cdots)\geq (\rho/2)\Vert A_i x_i^k-A_i x_i^{k+1}\Vert^2_2.\label{eq: lemma3}
\end{align}
\end{lemma}
\begin{proof}
\begin{align*}
&L_\rho(\cdots,x_{i-1}^{k+1},x_i^k,\cdots)-L_\rho(\cdots,x_i^{k+1},x_{i+1}^k,\cdots)\\&=f(\cdots,x_{i-1}^{k+1},x_i^k,\cdots)-f(\cdots,x_i^{k+1},x_{i+1}^k,\cdots)\\&+(y^k)^T(A_i x_i^k-A_i x_i^{k+1})+(\rho/2)\Vert \sum\nolimits_{j=1}^{i-1} A_j x_j^{k+1}+\sum\nolimits_{j=i}^n A_j x_j^{k}-z^k\Vert^2_2-(\rho/2)\Vert \sum\nolimits_{j=1}^{i} A_j x_j^{k+1}+\sum\nolimits_{j=i+1}^n A_j x_j^{k}-z^k\Vert^2_2\\
&=f(\cdots,x_{i-1}^{k+1},x_i^k,\cdots)-f(\cdots,x_i^{k+1},x_{i+1}^k,\cdots)\\&+(y^k)^T(A_i x_i^k-A_i x_i^{k+1})+(\rho/2)\Vert A_i x_i^k-A_i x_i^{k+1}\Vert^2_2+\rho(\sum\nolimits_{j=1}^{i}A_j x^{k+1}_j+\sum\nolimits_{j=i+1}^n A_jx_j^{k}-z^k)^T(A_ix_i^k-A_ix_i^{k+1})\\
&=f(\cdots,x_{i-1}^{k+1},x_i^k,\cdots)-f(\cdots,x_i^{k+1},x_{i+1}^k,\cdots)\\&+(A^T_iy^k+\rho A_i^T(\sum\nolimits_{j=1}^{i}A_j x^{k+1}_j+\sum\nolimits_{j=i+1}^n A_jx_j^{k}-z^k))^T(x_i^k-x_i^{k+1}) +(\rho/2)\Vert A_i x_i^k-A_i x_i^{k+1}\Vert^2_2.
\end{align*}
where the second equality follows from the cosine rule: $\Vert b+c\Vert^2-\Vert a+c\Vert^2=\Vert b-a\Vert^2+2(a+c)^T(b-a)$ with $a=A_ix_i^{k+1}$, $b=A_ix_i^k$ and $c=\sum\nolimits_{j=1}^{i-1} A_jx_j^{k+1}+\sum\nolimits_{j=i+1}^n A_jx_j^k-z^k$.\\
The optimality condition of $x^{k+1}_i$ leads to
\begin{align*}
    &0\in \partial_{x_i} L_\rho(\cdots,x^{k+1}_i,x^k_{i+1},\cdots)\\&=\partial_{x_i} f(\cdots,x^{k+1}_i,x^k_{i+1},\cdots)+A_i^Ty^k+\rho A_i^T(\sum\nolimits_{j=1}^{i}A_j x^{k+1}_j+\sum\nolimits_{j=i+1}^n A_jx_j^{k}-z^k)\\
    & -A_i^Ty^k-\rho A_i^T(\sum\nolimits_{j=1}^{i}A_j x^{k+1}_j+\sum\nolimits_{j=i+1}^n A_jx_j^{k}-z^k)\in \partial_{x_i} f(\cdots,x^{k+1}_i,x^k_{i+1},\cdots).
\end{align*}
  We have the following result according to the definition of subgradient
\begin{align*}
&f(\cdots,x_{i-1}^{k+1},x_i^k,\cdots)\\&\geq f(\cdots,x_{i}^{k+1},x_{i+1}^k,\cdots)+(-A_i^Ty^k-\rho A_i^T(\sum\nolimits_{j=1}^{i}A_j x^{k+1}_j+\sum\nolimits_{j=i+1}^n A_jx_j^{k}-z^k))^T(x_i^{k+1}-x_i^{k})\\&=f(\cdots,x_{i}^{k+1},x_{i+1}^k,\cdots)+(A_i^Ty^k+\rho A_i^T(\sum\nolimits_{j=1}^{i}A_j x^{k+1}_j+\sum\nolimits_{j=i+1}^n A_jx_j^{k}-z^k))^T(x_i^{k}-x_i^{k+1}).    
\end{align*}
Therefore, the lemma is proved.
\end{proof}
\begin{lemma}
\label{lemma: property 2 y}
If $\rho>2H$  so that $C_1=\rho/2-H/2-H^2/\rho>0$, then it holds that
\begin{align}
L_\rho(\cdots,x^{k+1}_n,z^k,y^k)-L_\rho(\cdots,x^{k+1}_n,z^{k+1},y^{k+1})\geq  C_1\Vert z^{k+1}-z^{k}\Vert^2_2. \label{eq: lemma4}  
\end{align}
\end{lemma}
\begin{proof}
\begin{align*}
    &L_\rho(x_1^{k+1},\cdots,x^{k+1}_n,z^k,y^k)-L_\rho(x_1^{k+1},\cdots,x^{k+1}_n,z^{k+1},y^{k+1})\\&=h(z^k)+(y^{k})^T (\sum\nolimits_{i=1}^{n} A_i x_i^{k+1}-z^{k})+(\rho/2)\Vert \sum\nolimits_{i=1}^{n} A_i x_i^{k+1}-z^k\Vert^2_2\\&-h(z^{k+1})-(y^{k+1})^T (\sum\nolimits_{i=1}^{n} A_i x_i^{k+1}-z^{k+1})-(\rho/2)\Vert \sum\nolimits_{i=1}^{n} A_i x_i^{k+1}-z^{k+1}\Vert^2_2\\&=h(z^k)-h(z^{k+1})+(y^k-y^{k+1})^T\sum\nolimits_{i=1}^{n} A_i x_i^{k+1}+(y^{k+1})^T z^{k+1}-(y^{k})^Tz^{k}+(\rho/2)\Vert \sum\nolimits_{i=1}^{n} A_i x_i^{k+1}-z^k\Vert^2_2\\&-(\rho/2)\Vert \sum\nolimits_{i=1}^{n} A_i x_i^{k+1}-z^{k+1}\Vert^2_2 \\&=h(z^k)-h(z^{k+1})+(y^k-y^{k+1})^T\sum\nolimits_{i=1}^{n} A_i x_i^{k+1}+(y^{k+1})^T z^{k+1}-(y^{k})^Tz^{k}\\&+(\rho/2)\Vert  z^{k+1}-z^k\Vert^2_2+\rho(z^{k+1}-\sum\nolimits_{i=1}^{n}A_ix^{k+1}_i)^T(z^k-z^{k+1})\\&\ \ \ \text{(cosine rule $\Vert a-b\Vert ^2-\Vert a-c\Vert^2=\Vert c-b\Vert ^2+2( c-a) ^T(b-c) $ where $a=\sum\nolimits_{i=1}^{n}A_i x_i^{k+1}$, $b=z^k$ and $c=z^{k+1}$.)}\\&=h(z^k)-h(z^{k+1})+(y^k-y^{k+1})^T\sum\nolimits_{i=1}^{n} A_i x_i^{k+1}+(y^{k+1})^T z^{k+1}-(y^{k})^Tz^{k}+(\rho/2)\Vert  z^{k+1}-z^k\Vert^2_2+(y^k-y^{k+1})^T(z^k-z^{k+1})\\&\ \ \ \text{(Because $y^{k+1}=y^k+\rho(\sum\nolimits_{i=1}^n A_ix^{k+1}_i-z^{k+1})$.)}\\&=h(z^k)-h(z^{k+1})+(y^k-y^{k+1})^T(\sum\nolimits_{i=1}^{n} A_i x_i^{k+1}+z^k-z^{k+1})+(y^{k+1})^T z^{k+1}-(y^{k})^Tz^{k}+(\rho/2)\Vert  z^{k+1}-z^k\Vert^2_2\\&=h(z^k)-h(z^{k+1})+(y^k-y^{k+1})^T(\sum\nolimits_{i=1}^{n} A_i x_i^{k+1}-z^{k+1})+(y^k-y^{k+1})^Tz^k+(y^{k+1})^T z^{k+1}-(y^{k})^Tz^{k}+(\rho/2)\Vert  z^{k+1}-z^k\Vert^2_2\\&=h(z^k)-h(z^{k+1})-(1/\rho)(y^k-y^{k+1})^T(y^k-y^{k+1})-(y^{k+1})^T(z^k-z^{k+1})+(\rho/2)\Vert  z^{k+1}-z^k\Vert^2_2\\&\ \ \ \text{(Because $y^{k+1}=y^k+\rho(\sum\nolimits_{i=1}^n A_ix^{k+1}_i-z^{k+1})$.)}\\&=h(z^k)-h(z^{k+1})-(y^{k+1})^T(z^{k}-z^{k+1})+(\rho/2)\Vert z^{k+1}-z^k\Vert^2_2-(1/\rho)\Vert y^{k+1}-y^k\Vert^2_2\\&=h(z^k)-h(z^{k+1})+\nabla h(z^{k+1})^T(z^{k+1}-z^{k})+(\rho/2)\Vert z^{k+1}-z^k\Vert^2_2-(1/\rho)\Vert y^{k+1}-y^k\Vert^2_2\ \ \ \text{(Lemma \ref{lemma: z optimal})}\\&\geq (-H/2)\Vert z^{k+1}-z^k\Vert^2_2+(\rho/2)\Vert z^{k+1}-z^k\Vert^2_2-(1/\rho)\Vert \nabla h(z^{k+1})-\nabla h(z^k)\Vert^2_2\\& \text{($-\nabla h(z)$ is Lipschitz differentiable, Lemma \ref{lemma: lemma1} and Lemma \ref{lemma: z optimal})}\\&\geq (-H/2)\Vert z^{k+1}-z^k\Vert^2_2+(\rho/2)\Vert z^{k+1}-z^k\Vert^2_2-(H^2/\rho)\Vert z^{k+1}-z^k\Vert^2_2\ \ \ \text{(Assumption \ref{ass:lipschitz differentiability and Coercivity})}\\&=C_1\Vert z^{k+1}-z^k\Vert^2_2.
\end{align*}
We choose $\rho>2H$ to make $C_1>0$.
\end{proof}
\begin{lemma}
\label{lemma:dual boundness}
 $ \forall k \in \mathbb{N}$, we have
 $
     \Vert y^{k+1}-y^k\Vert\leq H \Vert z^{k+1}-z^k\Vert
$.
\end{lemma}
\begin{proof}
\begin{align*}
   & \Vert y^{k+1}-y^k\Vert=\Vert \nabla h(z^{k+1})-\nabla h(z^{k})\Vert\ \ \ \text{(Lemma \ref{lemma: z optimal})}\leq  H\Vert z^{k+1}-z^k\Vert \ \ \ \text{(Assumption \ref{ass:lipschitz differentiability and Coercivity})}.
\end{align*}
\end{proof}
\section{Proofs of Lemmas \ref{lemma:objective descent}- \ref{lemma:objective bound}}
\label{sec: proof of objective descent and objective bound}
\begin{proof}[Proof of Lemma \ref{lemma:objective descent}]
This follows directly from Lemmas \ref{lemma: property 2 x} and \ref{lemma: property 2 y}.
\end{proof}
\begin{proof}[Proof of Lemma \ref{lemma:objective bound}]
There exists $z'$ such that $\sum\nolimits_{i=1}^n A_ix_i^{k}-z'=0$. Therefore, we have
\begin{align*}
    &f(x_1^k,\cdots,x_n^k)+h(z')\geq \min S > -\infty.
\end{align*}
where $S=\{ f(x_1,\cdots,x_n)+h(z):  \sum\nolimits_{i=1}^n A_i x_i-z=0\}$, which is the objective value of Problem \ref{prob:main problem}, and therefore bounded from below. Then we have
\begin{align*}
    &L_\rho(x_1^k,\cdots,x_n^k,z^k,y^k)\\&= f(x^k_1,\cdots,x^k_n)+h(z^k)+(y^k)^T(\sum\nolimits_{i=1}^n A_ix^k_i-z^k)+(\rho/2)\Vert \sum\nolimits_{i=1}^n A_ix^k_i-z^k\Vert^2\\&=f(x^k_1,\cdots,x^k_n)+h(z^k)+(y^k)^T(z^{'}-z^{k})+(\rho/2)\Vert \sum\nolimits_{i=1}^n A_ix^k_i-z^k\Vert^2 \  \text{$(\sum\nolimits_{i=1}^n A_ix_i^k-z^{'}=0)$}\\
    &=f(x^k_1,\cdots,x^k_n)+h(z^k)+(\nabla h(z^k))^T(z'-z^k)+(\rho/2)\Vert \sum\nolimits_{i=1}^n A_ix^k_i-z^k\Vert^2 \ \ \ \ \text{(Lemma \ref{lemma: z optimal})}\\&\geq f(x^k_1,\cdots,x^k_n)+h(z')+(\rho-H)/2\Vert \sum\nolimits_{i=1}^n A_ix^k_i-z^k\Vert^2_2  \ \ \ \ \text{(Lemmas  \ref{lemma: lemma1} and \ref{lemma: z optimal} ,$h(z)$ is Lipschitz differentiable)}\\&\geq \min S+(\rho-H)/2\Vert \sum\nolimits_{i=1}^n A_ix^k_i-z^k\Vert^2_2\geq\min S>-\infty. 
\end{align*}
Therefore, $L_\rho(x_1^k,\cdots,x_n^k,z^k,y^k)$ is  bounded from below.
\end{proof}

\section{Proofs of Theorems \ref{thero: theorem 2}-\ref{thero: theorem 3}}
\label{sec:convergence proof}
\begin{proof}[Proof of Theorem \ref{thero: theorem 2}] We show  residual convergence and objective convergence  based on Lemmas \ref{lemma:objective descent} and \ref{lemma:objective bound}. \\
From Lemma \ref{lemma:objective descent}, $L_\rho(x^k_1,\cdots,x^k_n,z^k,y^k)$ decreases monotonically, and $L_\rho(x^k_1,\cdots,x^k_n,z^k,y^k)$  is lower bounded by Lemma \ref{lemma:objective bound}. Therefore, $L_\rho(x^k_1,\cdots,x^k_n,z^k,y^k)$ is convergent because a monotone bounded sequence converges (Monotone Convergence Theorem). According to the continuity of $L_\rho$, we take $k\rightarrow\infty$ on the both sides of Inequality \eqref{ineq: objective descent} to obtain
\begin{align*}
    &\lim_{k\rightarrow \infty}(L_\rho(x_1^k,\cdots,x_n^k,z^k,y^k)-L_\rho(x_1^{k+1},\cdots,x_n^{k+1},z^{k+1},y^{k+1}))\\&\geq \lim_{k\rightarrow \infty} C_2(\Vert z^{k+1}-z^k\Vert^2_2+\sum\nolimits_{i=1}^n\Vert A_i(x_i^{k+1}-x_i^{k})\Vert^2_2).
\end{align*}
On one hand,  $L_\rho(x_1,\cdots,x_n,z,y)$ is convergent, so we have
\begin{align*}\lim_{k\rightarrow \infty} C_2(\Vert z^{k+1}-z^k\Vert^2_2+\sum\nolimits_{i=1}^n\Vert A_i(x_i^{k+1}-x_i^{k})\Vert^2_2)\leq 0.
\end{align*}
On the other hand, $C_2(\Vert z^{k+1}-z^k\Vert^2_2+\sum\nolimits_{i=1}^n\Vert A_i(x_i^{k+1}-x_i^{k})\Vert^2_2$ is nonnegative, so we get
\begin{align*}
    \lim_{k\rightarrow \infty} C_2(\Vert z^{k+1}-z^k\Vert^2_2+\sum\nolimits_{i=1}^n\Vert A_i(x_i^{k+1}-x_i^{k})\Vert^2_2)=0.
\end{align*}
This suggests that $ \lim_{k\rightarrow \infty}  (z^{k+1}-z^k)=0$ and $ \lim_{k\rightarrow \infty}  A_i(x_i^{k+1}-x_i^{k})=0(i=1,\cdots,n)$. Moreover, by Lemma \ref{lemma:dual boundness}, $\lim_{k\rightarrow \infty} \Vert y^{k+1}-y^k\Vert\leq H\lim_{k\rightarrow \infty} \Vert z^{k+1}-z^k\Vert=0$. So we have $\lim_{k\rightarrow \infty} (y^{k+1}-y^k)=0$.\\
\indent a). For residual convergence, by the Line 8 of Algorithm \ref{algo:proposed ADMM}, we have
\begin{align*}
    &\lim\nolimits_{k\rightarrow\infty}r^{k}=\lim\nolimits_{k\rightarrow\infty}(y^{k}-y^{k-1})/\rho=0.
\end{align*}
\indent b). For objective convergence, since 
\begin{align*}
    L_\rho(x^k_1,\cdots,x^k_n,z^k,y^k)&=F(x^k_1,\cdots,x^k_n,z^k,y^k)+(y^k)^T r^k+(\rho/2)\Vert r^k\Vert^2_2
\end{align*}
and $L_\rho(x^k_1,\cdots,x^k_n,z^k,y^k)$ is convergent, $r^k$ converges to 0 and $y^k$ is bounded, then $F(x^k_1,\cdots,x^k_n,z^k,y^k)$ is also convergent.
\end{proof}
\begin{proof}[Proof of Theorem \ref{thero: Nash point theorem}]
\indent Obviously $\lim_{k\rightarrow\infty}(z^{k+1}-z^{k})=0$ and $\lim_{k\rightarrow\infty}(y^{k+1}-y^{k})=0$ from the proof of Theorem \ref{thero: theorem 2}. In order to prove this theorem, we firstly prove that $\lim_{k\rightarrow\infty} (x^{k+1}_i-x^{k}_i)=0(i=1,\cdots,n)$ if either of two assumptions holds, then prove that any limit point $(x^*_1,\cdots,x^*_n,z^*)$ is a feasible Nash point of Problem 1. \\
\indent (a). Suppose $A_i(i=1,\cdots,n)$ have full rank. Because $\lim_{k\rightarrow\infty} A_i(x^{k+1}_i-x^{k}_i)=0$ from the proof of Theorem \ref{thero: theorem 2}, then obviously $\lim_{k\rightarrow\infty} (x^{k+1}_i-x^{k}_i)=0$ \cite{lin2015sublinear}.\\
\indent (b). Suppose  $F$ is strongly convex with regard to $x_i$. Because  $L_\rho(x_1,\cdots,x_n,z,y)=F(x_1,\cdots,x_n,z)+y^T(\sum\nolimits A_i x_i-z)+(\rho/2)\Vert \sum\nolimits A_i x_i-z \Vert^2_2$, $F(x_1,\cdots,x_n,z)$, and $y^T(\sum\nolimits A_i x_i-z)+(\rho/2)\Vert \sum\nolimits A_i x_i-z \Vert^2_2$ are strongly convex,  $L_\rho$ is also strongly convex regard to $x_i$ \cite{merentes2010remarks} with the assumed constant $D_i>0$. We have
\begin{align*}
    L_\rho(x^{k+1}_1,\cdots,x^{k+1}_{i-1},x^k_i,x^{k}_{i+1},\cdots, x^k_n,z^k,y^k)&\geq L_\rho(x^{k+1}_1,\cdots,x^{k+1}_{i-1},x^{k+1}_i,x^{k}_{i+1},\cdots,x^k_n,z^k,y^k)+(v_i^{k+1})^T(x_i^{k}-x_i^{k+1})\\&+(D_i/2)\Vert x_i^{k+1}-x_i^{k} \Vert^2_2
\end{align*}
where $ \forall v_i^{k+1}\in \partial_{x_i} L_\rho(x^{k+1}_1,\cdots,x^{k+1}_{i-1},x^{k+1}_i,x^{k}_{i+1},\cdots,x^k_n,z^k,y^k)$. The optimality condition of $x^{k+1}_i$ leads to \\$0 \in \partial_{x_i} L_\rho(x^{k+1}_1,\cdots,x^{k+1}_{i-1},x^{k+1}_i,x^{k}_{i+1},\cdots,x^k_n,z^k,y^k)$. Therefore, we have
\begin{align}
    &L_\rho(x^{k+1}_1,\cdots,x^{k+1}_{i-1},x^k_i,x^k_{i+1},\cdots, x^k_n,z^k,y^k)\geq L_\rho(x^{k+1}_1,\cdots,x^{k+1}_{i-1},x^{k+1}_i,x^{k}_{i+1},\cdots,x^k_n,z^k,y^k)+(D_i/2)\Vert x_i^{k+1}-x_i^{k} \Vert^2_2 \label{ineq: strong convexity}
\end{align}
We sum up Inequality \eqref{ineq: strong convexity} from $i=1,\cdots,n$ and Inequality \eqref{eq: lemma4}  to obtain
\begin{align}
 L_\rho(x^{k}_1,\cdots, x^k_n,z^k,y^k) -  L_\rho(x^{k+1}_1,\cdots, x^{k+1}_n,z^{k+1},y^{k+1})\geq \sum\nolimits_{i=1}^n (D_i/2)\Vert x_i^{k+1}-x_i^{k} \Vert^2_2+C_1\Vert z^{k+1}-z^k\Vert^2_2 \label{ineq: strong convexity descent}
\end{align}
where $C_1>0$ by Lemma \ref{lemma: property 2 y} if $\rho>2H$. According to the continuity of $L_\rho$, we take $k\rightarrow\infty$ on the both sides of Inequality \eqref{ineq: strong convexity descent} to obtain
\begin{align*}
    &\lim_{k\rightarrow \infty}(L_\rho(x_1^k,\cdots,x_n^k,z^k,y^k)-L_\rho(x_1^{k+1},\cdots,x_n^{k+1},z^{k+1},y^{k+1}))\geq \lim_{k\rightarrow \infty} (\sum\nolimits_{i=1}^n (D_i/2)\Vert x_i^{k+1}-x_i^{k} \Vert^2_2+C_1\Vert z^{k+1}-z^k\Vert^2_2)
\end{align*}
On one hand,  $L_\rho(x_1,\cdots,x_n,z,y)$ is convergent, so we have
\begin{align*}\lim_{k\rightarrow \infty} (\sum\nolimits_{i=1}^n (D_i/2)\Vert x_i^{k+1}-x_i^{k} \Vert^2_2+C_1\Vert z^{k+1}-z^k\Vert^2_2)\leq 0
\end{align*}
On the other hand, $\sum\nolimits_{i=1}^n (D_i/2)\Vert x_i^{k+1}-x_i^{k} \Vert^2_2+C_1\Vert z^{k+1}-z^k\Vert^2_2$ is nonnegative, so we get
\begin{align*}
    \lim_{k\rightarrow \infty} (\sum\nolimits_{i=1}^n (D_i/2)\Vert x_i^{k+1}-x_i^{k} \Vert^2_2+C_1\Vert z^{k+1}-z^k\Vert^2_2)=0
\end{align*}
This suggests that $ \lim_{k\rightarrow \infty}  (x_i^{k+1}-x_i^{k})=0(i=1,\cdots,n)$ and $ \lim_{k\rightarrow \infty}  (z^{k+1}-z^k)=0$. \\
\indent Therefore,  $ \lim_{k\rightarrow \infty}  (x_i^{k+1}-x_i^{k})=0(i=1,\cdots,n)$ if either of two assumptions holds.
Because $(x^k_1,\cdots, x^k_n, z^k,y^k)$ is bounded, there exists a subsequence  $(x_1^s,\cdots,x_n^s,z^s,y^s)$ such that $(x_1^s,\cdots,x_n^s,z^s,y^s)\rightarrow (x_1^*,\cdots,x_n^*,z^*,y^*)$ where $(x_1^*,\cdots,x_n^*,z^*,y^*)$ is a limit point. Because $\lim_{s\rightarrow \infty}  (x_i^{s+1}-x_i^{s})=0(i=1,\cdots,n)$, $ \lim_{s\rightarrow \infty}  (z^{s+1}-z^s)=0$ and $ \lim_{s\rightarrow \infty}  (y^{s+1}-y^s)=0$, we have $(x_1^{s+1},\cdots,x_n^{s+1},z^{s+1},y^{s+1})\rightarrow (x_1^*,\cdots,x_n^*,z^*,y^*)$. Now we prove that the limit point $(x_1^*,\cdots,x_n^*,z^*)$ is a feasible Nash point of Problem 1.\\
\indent For feasibility, since  $\lim_{k\rightarrow\infty} r^k=\lim_{k\rightarrow\infty} \sum\nolimits_{i=1}^n A_ix^k_i-z^k=0$, so for the subsequence $(x_1^s,\cdots,x_n^s,z^s,y^s)\rightarrow (x_1^*,\cdots,x_n^*,z^*,y^*)$, we have $\lim_{s\rightarrow\infty} r^s=\lim_{s\rightarrow\infty} (\sum\nolimits_{i=1}^n A_ix^s_i-z^s)=0$ then $\sum\nolimits_{i=1}^n A_ix^*_i-z^*=0$.
\\\indent For the Nash point, we obtain the following according to the optimality conditions of $x^{s+1}_i(i=1,\cdots,n)$ and $z^{s+1}$ in Equations \eqref{eq:update x} and \eqref{eq:update z}, respectively.
\begin{align*}
    &L_\rho(x^{s+1}_1,\cdots,x^{s+1}_{i-1},x^{s+1}_{i},x^{s}_{i+1},\cdots,x^{s}_{n},z^s,y^s)\leq L_\rho(x^{s+1}_1,\cdots,x^{s+1}_{i-1},x_{i},x^{s}_{i+1},\cdots,x^{s}_{n},z^s,y^s), \\& \forall (x^{s+1}_1,\cdots,x^{s+1}_{i-1},x_{i},x^{s}_{i+1},\cdots,x^{s}_{n},z^s) \in dom(F)\\
&L_\rho(x^{s+1}_1,\cdots,x^{s+1}_{n},z^{s+1},y^s)\leq L_\rho(x^{s+1}_1,\cdots,x^{s+1}_{n},z,y^s), \ \forall (x^{s+1}_1,\cdots,x^{s+1}_{n},z)\in dom(F)
\end{align*}
According to the continuity of $L_\rho$, we take $s \rightarrow\infty$ on the both sides of two inequalities. Because $(x_1^{s},\cdots,x_n^{s},z^{s},y^s)\rightarrow (x_1^*,\cdots,x_n^*,z^*,y^*)$ and $(x_1^{s+1},\cdots,x_n^{s+1},z^{s+1},y^{s+1})\rightarrow (x_1^*,\cdots,x_n^*,z^*,y^*)$, we have
\begin{align*}
    &L_\rho(x^*_1,\cdots,x^*_{n},z^*,y^*)\leq L_\rho(x^*_1,\cdots,x^*_{i-1},x_{i},x^*_{i+1},\cdots,x^*_{n},z^*,y^*), \ \forall (x^*_1,\cdots,x^*_{i-1},x_{i},x^*_{i+1},\cdots,x^*_{n},z^*) \in dom(F)\\
&L_\rho(x^*_1,\cdots,x^*_{n},z,y^*)\leq L_\rho(x^*_1,\cdots,x^*_{n},z^*,y^*), \ \forall (x^*_1,\cdots,x^*_{n},z)\in dom(F)
\end{align*}
Here $\forall (x^*_1,\cdots,x^*_{i-1},x_{i},x^*_{i+1},\cdots,x^*_{n},z^*) \in dom(F)$ and $\forall (x^*_1,\cdots,x^*_{n},z)\in dom(F)$ mean  $\forall x_i \ s.t. \ \sum\nolimits_{j=1,j\neq i}^n A_jx^*_j+A_ix_i-z^*=0$ and $\forall z \ s.t. \ \sum\nolimits_{j=1}^n A_jx^*_j-z=0$, respectively. Using the fact that $(x^*_1,\cdots,x^*_n,z^*)$ is feasible in Problem 1, we obtain $L_\rho(x^*_1,\cdots,z^*,y^*)=F(x^*_1,\cdots,z^*)$, $L_\rho(x^*_1,\cdots,x^*_{i-1},x_{i},x^*_{i+1},\cdots,x^*_{n},z^*,y^*)=F(x^*_1,\cdots,x^*_{i-1},x_{i},x^*_{i+1},\cdots,x^*_{n},z^*)$ and $L_\rho(x^*_1,\cdots,x^*_{n},z,y^*)=F(x^*_1,\cdots,x^*_{n},z)$. Therefore, we prove that $(x^*_1,\cdots,x^*_n,z^*)$ is a feasible Nash point of $F$ defined in Problem 1.
\end{proof} 

\begin{proof}[Proof of Theorem \ref{thero: theorem 3}]
To prove this theorem, we will first show that $u_k$ satisfies two conditions: (1). $u_k\geq u_{k+1}$. (2). $\sum\nolimits_{k=0}^\infty u_k$ is bounded.  We then conclude the convergence rate of $o(1/k)$ based on these two conditions. Specifically, first, we have
\begin{align*}
    u_k&=\min\nolimits_{0\leq l\leq k}(\Vert z^{l+1}-z^l\Vert^2_2+\sum\nolimits_{i=1}^n\Vert A_i(x_i^{l+1}-x_i^{l})\Vert^2_2)\\&\geq \min\nolimits_{0\leq l\leq k+1}(\Vert z^{l+1}-z^l\Vert^2_2+\sum\nolimits_{i=1}^n\Vert A_i(x_i^{l+1}-x_i^{l})\Vert^2_2)\\&=u_{k+1}
\end{align*}
Therefore $u_k$ satisfies the first condition. Second,
\begin{align*}
    &\sum\nolimits_{k=0}^\infty u_k\\&=\sum\nolimits_{k=0}^\infty \min\nolimits_{0\leq l\leq k}(\Vert z^{l+1}-z^l\Vert^2_2+\sum\nolimits_{i=1}^n\Vert A_i(x_i^{l+1}-x_i^{l})\Vert^2_2)\\&\leq \sum\nolimits_{k=0}^\infty (\Vert z^{k+1}-z^k\Vert^2_2+\sum\nolimits_{i=1}^n\Vert A_i(x_i^{k+1}-x_i^{k})\Vert^2_2)\\&\leq (L_\rho(x^0_1,\cdots,x^0_n,z^0,y^0)-L^*_\rho)/C_2 \ \text{(Lemma \ref{lemma:objective descent})}
\end{align*}
where $L^*_\rho=\lim_{k\rightarrow \infty}L_\rho(x^k_1,\cdots,x^k_n,z^k,y^k)$.
So $\sum\nolimits_{k=0}^\infty u_k$ is bounded and $u_{k}$ satisfies the second condition. Finally, it has been proved that the sufficient conditions of convergence rate $o(1/k)$ are: (1) $u_k\geq u_{k+1}$, and (2) $\sum\nolimits_{k=0}^\infty u_k$ is bounded, and (3) $u_k\geq0$ (Lemma 1.2 in \cite{deng2017parallel}). Since we have proved the first two conditions and the third one $u_k \geq 0$ is obvious, the convergence rate of $o(1/k)$ is proven. 
\end{proof}
\section{Algorithms for Applications}
\label{sec: other applications}
\subsection{Weakly-constrained Multi-task Learning}
\label{sec:multi-task learning}
Applying the proposed ADMM to solve the problem in Equation \eqref{prob: multi-task learning}, we get Algorithm \ref{algo: algorithm 3}.
Specifically, Lines 4-9 update primal variables $w_i(i=1,\cdots,n)$ and $z_i(i=1,\cdots,n)$, Line 10 updates the dual variable $y_i(i=1,\cdots,n)$.
\begin{algorithm} 
\caption{The Proposed ADMM to Solve Equation \eqref{prob: multi-task learning}.} 
\begin{algorithmic}[1]
\label{algo: algorithm 3}
\scriptsize
\STATE Denote $z=[z_1;\cdots;z_n]$,$y=[y_1;\cdots;y_n]$.
\STATE Initialize $\rho$, $k=0$.
\REPEAT
\STATE Update $w_1^{k+1}$ by Equation \eqref{eq:update w1}. 
\FOR{i=2 to n-1}
\STATE Update $w^{k+1}_i$ by Equation \eqref{eq:update wi}.\\
\ENDFOR
\STATE Update $w^{k+1}_n$ by Equation \eqref{eq:update wn}.\\
\STATE Update $z_i^{k+1}$ by Equation \eqref{eq:update zi} in parallel.
\STATE $y_i^{k+1}\leftarrow y_i^k+\rho(w^{k+1}_i-z^{k+1}_i)\ (i=1,\cdots,n)$ in parallel.
\STATE $k\leftarrow k+1$.\\
\UNTIL convergence.
\STATE Output $w_i(i=1,\cdots,n),z$.
\end{algorithmic}
\end{algorithm}
\\\indent All subproblems are detailed as follows:\\
\textbf{1. Update $w^{k+1}$}
\\\indent The $w^{k+1}_i(i=1,\cdots,n)$ are updated as follows:
\begin{align}
    &w_1^{k+1} \leftarrow \arg\min\nolimits_{w_1} Loss_1(w_1)+\lambda_1\sum\nolimits_{j=1}^{m}c_1(w_{1,j}w^k_{2,j})+(\rho/2)\Vert w_1-z^k_1+y^k_1/\rho\Vert^2_2. \label{eq:update w1}\\
    &w_i^{k+1} \leftarrow \arg\min\nolimits_{w_i} Loss_i(w_i)+\lambda_1\sum\nolimits_{j=1}^{m}c_1(w_{i,j}w^k_{i+1,j})+\lambda_1\sum\nolimits_{j=1}^{m}c_1(w^{k+1}_{i-1,j}w_{i,j})+(\rho/2)\Vert w_i-z^k_i+y^k_i/\rho\Vert^2_2. \label{eq:update wi}
\\&w_n^{k+1} \leftarrow \arg\min\nolimits_{w_n} Loss_n(w_n)+\lambda_1\sum\nolimits_{j=1}^{m}c_1(w^{k+1}_{n-1,j}w_{n,j})+(\rho/2)\Vert w_n-z^k_n+y^k_n/\rho\Vert^2_2. \label{eq:update wn}
\end{align}
They can be solved by the Iterative Soft Thresholding Algorithm (ISTA) \cite{beck2009fast}. Take $w_1^{k+1}$ as an example, The ISTA leads to 
\begin{align*}
    w_1^{t+1}\leftarrow \lambda_1\sum\nolimits_{j=1}^{m}c_1(w_{1,j}w^k_{2,j})+1/(2\eta)\Vert w_1-(w_1^{t}-\eta\nabla\phi(w_1^{t}))\Vert^2_2.
\end{align*}
where $w_1^{t}$ is the $t$-th iteration in the ISTA, $\eta>0$ is a learning rate, $\phi(w^t_1)=Loss_1(w^t_1)+\rho/2\Vert w^t_1-z^k_1+y^k_1/\rho\Vert^2_2$. For each entry of $w^{t+1}_{1,j}(j=1,2,\cdots,m)$, we have the following closed-form solution as follows:\\
1). If $w^{t+1}_{1,j}w^k_{2,j}\leq 0$, then $w^{t+1}_{1,j}\leftarrow(w_{1,j}^{t}-\eta\nabla_j\phi(w_1^{t}))/(2\eta\lambda_1w^2_{2,j}+1)$.\\
2). If $w^{t+1}_{1,j}w^k_{2,j}\geq 0$, then
$w_{1,j}^{t+1}\leftarrow w_{1,j}^{t}-\eta\nabla_j\phi(w_1^{t})$. where $\nabla_j\phi(w_1^{t})$ is the $j$-th entry of $\nabla\phi(w_1^{t})$.\\
\textbf{2. Update $z^{k+1}$}
\\\indent The $z^{k+1}_i(i=1,\cdots,n)$ are updated as follows:
\begin{align}
    z_i^{k+1} \leftarrow \arg\min\nolimits_{z_i}\Omega_i(z_i)+(\rho/2)\Vert w^{k+1}_i-z_i+y^k_i/\rho\Vert^2_2 (i=1,\cdots,n).
    \label{eq:update zi}
\end{align}
For $\ell_1$ or $\ell_2$ regularization, they have closed-form solutions.
\subsection{Learning with Signed-Network Constraints}
\label{sec:muli-lingual}
Applying proposed ADMM to solve the problem in Equation \eqref{prob: muli-lingual}, we get Algorithm \ref{algo: algorithm 4}. Specifically, Lines 4-7 update primal variables $\beta_i(i=1,\cdots,n)$ and $z$, Line 8 updates the dual variable $y$. 
\begin{algorithm} 
\caption{The Proposed ADMM to Solve Equation \eqref{prob: muli-lingual}.} 
\begin{algorithmic}[1]
\scriptsize
\label{algo: algorithm 4}
\STATE Denote $z=[z_1;\cdots;z_n]$,$y=[y_1;\cdots;y_n]$. 
\STATE Initialize $\rho$, $k=0$.
\REPEAT
\FOR{i=1 to n}
\STATE  Update $\beta^{k+1}_i$ by Equation \eqref{eq:update beta}.
\ENDFOR
\STATE Update $z_i^{k+1}(i=1,\cdots,n)$ by Equation \eqref{eq:update zi2} in parallel.
\STATE $y_i^{k+1}\leftarrow y_i^k+\rho(\beta^{k+1}_i-z_i^{k+1})\ (i=1,\cdots,n)$ in parallel.
\STATE $k\leftarrow k+1$.\\
\UNTIL convergence.
\STATE Output $\beta_i(i=1,\cdots,n),z$.
\end{algorithmic}
\end{algorithm}
All subproblems are detailed as follows:\\
\textbf{1. Update $\beta^{k+1}$}
\\\indent The $\beta^{k+1}_i(i=1,\cdots,n)$ are updated as follows:
\begin{align}
    \nonumber\beta_i^{k+1} &\leftarrow \arg\min\nolimits_{\beta_i} Loss(\cdots,\beta^{k+1}_{i-1},\beta_i,\beta^{k}_{i+1},\cdots)+\lambda_2(\sum\nolimits_{(\beta_i,\beta^{k+1}_j)\in E_s,j<i} c_2(\beta_i,\beta^{k+1}_j)+\sum\nolimits_{(\beta_i,\beta^{k}_j)\in E_s,j>i} c_2(\beta_i,\beta^{k}_j)\\&+\sum\nolimits_{(\beta_i,\beta^{k+1}_q)\in E_d,q<i} c_3(\beta_i,\beta^{k+1}_q)+\sum\nolimits_{(\beta_i,\beta^{k}_q)\in E_d,q>i} c_3(\beta_i,\beta^{k}_q)+(\rho/2)\Vert \beta_i-z^k_i+y^k_i/\rho\Vert^2_2).\label{eq:update beta}
\end{align}
\indent Similar to updating $w^{k+1}_i$ in Algorithm \ref{algo: algorithm 3}, they can be solved efficiently by the ISTA \cite{beck2009fast}.\\
\textbf{2. Update $z^{k+1}$}
\\\indent The $z^{k+1}_i(i=1,\cdots,n)$ are updated as follows:
\begin{align}
    z_i^{k+1} \leftarrow \arg\min\nolimits_{z_i}\omega_i(z_i)+(\rho/2)\Vert \beta^{k+1}_i-z_i+y^k_i/\rho\Vert^2_2 (i=1,\cdots,n).\label{eq:update zi2}
\end{align}
\indent Similar to updating $z^{k+1}_i$ in Algorithm \ref{algo: algorithm 3}, they usually have closed-form solutions.
\end{appendix}
\end{document}